\documentclass[a4paper]{article}
\usepackage[latin1]{inputenc} %
\usepackage[T1]{fontenc} %
\usepackage{RR}
\usepackage{hyperref}

\usepackage[all,dvips]{xy}
\usepackage{graphicx}
\usepackage{mathrsfs}
\usepackage[mathcal]{euscript}
\usepackage{amsmath,amssymb,amsthm}

\RRdate{Decembre 2009}
\RRauthor{% les auteurs
Etienne Farcot\thanks{Virtual Plants INRIA, UMR DAP, CIRAD, TA A-96/02, 34398 Montpellier Cedex 5, France, farcot@sophia.inria.fr}
\and
Jean-Luc Gouz\'e\thanks{COMORE INRIA, Unit\'{e} de recherche Sophia Antipolis Méditerranée , 2004 route des Lucioles, BP 93, 06902 Sophia Antipolis, France, gouze@sophia.inria.fr}
}
\authorhead{Farcot \& Gouz\'e}
\RRtitle{Commande qualitative de solutions périodiques de systèmes affines par morceaux ; application aux réseaux génétiques}
\RRetitle{Qualitative control of periodic solutions in piecewise affine systems; application to genetic networks}
\titlehead{Control of limit cycles in PWA gene networks}
\RRresume{Les systèmes hybrides, en particulier  affines par morceaux (APM), sont souvent employés comme modèles de réseaux génétiques. Dans ce rapport nous approfondissons des travaux antérieurs sur la commande de tels systèmes, utilisant également des résultats récents garantissant l'existence et l'unicité de cycles limites, sur la seule base d'une abstraction discrète du système et de sa structure d'interaction. L'objectif est de contr\^oler le graphe de transitions d'états du système APM pour obtenir un comportement périodique, ce qui est une propriété très importante de nombreux systèmes biologiques. Nous montrons comment commander l'apparition ou la suppression d'un cycle limite unique, par une action qualitative sur les taux de dégradation d'un système APM, aussi bien par commande statique que dynamique, c'est-à-dire par le couplage adéquat d'un sous-réseau contr\^{o}leur. Ceci est illustré sur deux réseaux de gènes classiques, présentant une structure de boucles de retro-action imbriquées.}

\RRabstract{Hybrid systems, and especially piecewise affine (PWA) systems, are often used to model gene regulatory networks. In this paper we elaborate on previous work about control problems for this class of models, using also some recent results guaranteeing the existence and uniqueness of limit cycles, based solely on a discrete abstraction of the system and its interaction structure. Our aim is  to control the transition graph of the PWA system to obtain an oscillatory behaviour, which is indeed of primary functional importance in  numerous biological networks; we show how it is possible to control the appearance or disappearance of a unique stable limit cycle by hybrid qualitative action on the degradation rates of the PWA system, both by static and dynamic feedback, i.e. the adequate coupling of a controlling  subnetwork. This is illustrated on two classical gene network modules, having the structure of mixed feedback loops.}
\RRmotcle{R\'eseaux génétiques, Commande en feedback, Linéaire par morceaux, Solutions périodiques}
\RRkeyword{Gene Networks, Feedback Control, Piecewise Linear, Periodic Solutions}
\RRprojets{Virtual Plants et Comore}
\RRdomaine{1} % cas du domaine numero 1
\RRtheme{Observation, mod\'elisation et commande pour le vivant}
\URSophia % pour ceux qui sont au Sud.
\RCSophia % Sophia Antipolis M\'editerran\'ee

\newtheorem{theorem}{Theorem}
\newtheorem{prop}{Proposition}
\newtheorem{lemma}{Lemma}

\newcommand{\NN}[1]{\{1,\cdots, #1\}}
\newcommand{\Nm}[1]{\{0,\cdots, #1-1\}}
\begin{document}
\RRNo{7130}
\makeRR   % cas d'un rapport de recherche
%% \makeRT % cas d'un rapport technique.
%% a partir d'ici, chacun fait comme il le souhaite

\section{Introduction}
Gene regulatory networks often display both robustness and steep, almost switch-like, response to transcriptional control. This motivates the use of an approximation of these response laws by piecewise affine differential (PWA) equations, to build hybrid models of genetic networks. PWA systems are affine in each rectangular domain (or box) of the state space. They have been introduced in the 1970's by Leon Glass~\cite{glasstopkin} to model genetic networks. It has led to a long series of works by different authors, dealing with various aspects of these equations, {\it e.g.} \cite{casey,edwardetc,farcot06,glasstopkin,gouzesari}. They have been used also as models of concrete biological systems~\cite{bacillus}.%,batt}.

From an hybrid system point of view, the behaviour of PWA systems can be described by a transition graph, which is an abstraction (in the hybrid system sense) of the continuous system. This transition graph describes the possible transition between the boxes. It is also possible to check properties of the transition graph by model checking techniques \cite{batt2005vqm}.

Nowadays, the extraordinary development of biomolecular experimental techniques makes it possible to design and implement control laws in the cell system. % \cite{}.
The authors have recently developed a mathematical framework for controlling gene networks with hybrid controls \cite{FG08}; these controls are defined on each box.
It is easy to see that this amounts to change the transition graph to obtain the desired one.

From another point of view, more oriented towards dynamical systems, it is also possible to obtain results concerning the limit cycles in PWA systems (see \cite{glasspastern2} and the recent generalisation in \cite{FG06a}). For example, one can show that a simple negative loop in dimension greater that two produces a unique stable limit cycle. It is clear that biological oscillations play a fundamental role in the cell (\cite{letnat1}).

Our aim in this paper is to control PWA systems to make a single stable limit cycle appear or disappear. To fulfil that goal, after some recalls concerning the PWA systems, we use  some results enabling to deduce the existence of a single stable limit cycle in the state space from a periodic behaviour in a box sequence (section \ref{sec-cyc}), then  the results on the control of the transition graph in the space of boxes (section \ref{sec-con}), to obtain our main results illustrated by 2 examples (section \ref{sec-ex}).

Related works on  control aspects concern the affine or multi-affine hybrid systems (\cite{HabJVS2004, belta2006ccn}). The authors derive sufficient conditions for driving all the solutions out of some box. Other related works study the existence of limit cycles in the state space \cite{glasspastern2,Edw2009}. We are not aware of works linking control theory and limit cycle for this class of hybrid systems. 

\section{Piecewise affine models}
\label{sec-pwa}
\subsection{General formulation}
This section contains basic definitions and notations for piecewise affine models \cite{glasstopkin,edwards,farcot06,inria}. The general form of these models can be written as:
\begin{equation}\label{eq-genenet}
\frac{dx}{dt} = \kappa(x) - \Gamma(x) x
\end{equation}
The variables $(x_1\dots x_n)$ represent levels of expression of $n$ interacting genes, meaning in general concentrations of 
the mRNA or protein they code for. We will simply call {\it genes} the $n$ network elements in  the following. 
Since gene transcriptional regulation is often considered to follow a steep sigmoid law, an approximation by a step function has been 
proposed to model the response of a gene (i.e. its rate of transcription) to the activity of its regulators~\cite{glasstopkin}. 
We use the notation:
\[
\left\{\begin{array}{lcl}
{\sf s}^+(x,\theta) & = & 0\quad\text{if } x<\theta,\\
{\sf s}^+(x,\theta) & = & 1\quad\text{if } x>\theta,
\end{array}\right.
\]
This describes an effect of activation, whereas ${\sf s}^-(x,\theta)=1-{\sf s}^+(x,\theta)$ represents inhibition. Unless further precision are given, we leave this function undefined at its threshold value $\theta$.\\
The maps $\kappa:\mathbb{R}_+^n\to\mathbb{R}^n_+$ and $\Gamma:\mathbb{R}_+^n\to\mathbb{R}^{n\times n}_+$ in (\ref{eq-genenet}) are usually multivariate polynomials (in general multi-affine), applied to step functions of the form ${\sf s}^\pm(x_i,\theta_{i})$, where for each $i\in\NN{n}$ the threshold values belong to a finite set
\begin{equation}\label{eq-thresh}
\Theta_i=\{\theta_{i}^0,\dots,\theta_{i}^{q_i}\}.
\end{equation}
We suppose that the thresholds are ordered (i.e. $\theta_{i}^{j}<\theta_{i}^{j+1}$), and the extreme values $\theta_i^0=0$ and $\theta_i^{q_i}$ represent the range of values taken by $x_i$ rather than thresholds. \\
$\Gamma$ is a diagonal matrix whose diagonal entries $\Gamma_{ii}=\gamma_i$, are degradation rates of variables in the system.
Obviously, $\Gamma$ and the production rate $\kappa$ are piecewise-constant, taking fixed values in the rectangular domains obtained as Cartesian products of intervals bounded by values in the threshold sets (\ref{eq-thresh}). These rectangles, or {\it boxes}, or {\it regular domains}~\cite{plahmestl98,inria}, are well characterised by integer vectors: we will often refer to a box $\mathcal{D}_a=\prod_i (\theta_i^{a_i-1},\,\theta_i^{a_i})$ by its lower-corner index $a=(a_1\!-\!1\dots a_n\!-\!1)$. The set of boxes is then isomorphic to
\begin{equation}\label{eq-A}
\mathcal{A} = \prod_{i=1}^n\Nm{q_i},
\end{equation}
Also, the following pairs of functions will be convenient notations: $\theta_i^\pm:\mathcal{A}\to\Theta_i$, $\theta_i^-(a)=\theta_{i}^{a_i-1}$ 	and  $\theta_i^+(a)=\theta_{i}^{a_i}$.\\ 
Let us call {\it singular domains} the intersections of closure of boxes with threshold hyperplanes, where some $x_i\in\Theta_i\setminus\{\theta_i^0,\theta_i^{q_i}\}$. On these domains, the right-hand side of (\ref{eq-genenet}) is undefined in general. Although the notion of Filippov solution provides a generic solution to this problem~\cite{gouzesari}, in the case where the normal of the vector field has the same sign on both side of these singular hyperplanes, it is more simply possible to extend the flow by continuity. In the remaining of this paper, we will only consider trajectories which do not meet any singular domain, a fact holding necessarily in absence of auto-regulation, i.e. when no $\kappa_i$ depends on $x_i$. This leads to the following hypothesis:
\[
\forall i\in\NN{n},\; \kappa_i \text{ and } \gamma_i\text{ do not depend on }x_i. \qquad\quad \textbf{(H1)}
\] 

%\subsection{Regular dynamics}\label{sec-regdyn}
On any regular domain of index $a\in\mathcal{A}$, the rates $\kappa=\kappa(a)$ and $\Gamma=\Gamma(a)$ are constant, and thus equation~(\ref{eq-genenet}) is affine. Its solution is explicitly known, for each coordinate $i$~:
\begin{equation}\label{eq-flow}
\varphi_i(x,t)=x_i(t) = \phi_i(a) + e^{-\gamma_i t}\left(x_i -\phi_i(a)\right),
\end{equation}
where $t\in \mathbb{R}_+$ is such that $x(t)\in{\mathcal{D}_a}$, and 
\[
\phi(a) =
\left(\phi_1(a)\cdots\phi_n(a)\right)=\left(\frac{\kappa_1(a)}{\gamma_1(a)}\cdots\frac{\kappa_n(a)}{\gamma_n(a)}\right).
\]
It is clearly an attractive equilibrium of the flow~(\ref{eq-flow}). It will be called {\it focal point} in the following for reasons we explain now. Let us first make the generic assumption that no focal point lies on a singular domain:
\[
\forall a\in\mathcal{A},\quad \phi(a)\in \bigcup_{a'\in\mathcal{A}}\mathcal{D}_{a'}. \qquad\qquad\quad \textbf{(H2)}
\] 
Then, if $\phi(a)\in\mathcal{D}_a$, it is an asymptotically stable steady state of system~(\ref{eq-genenet}). Otherwise, the flow will reach the (topological) boundary
$\partial{\mathcal{D}_a}$ in finite time. At this point, the value of $\kappa$ (and thus, of $\phi$) changes, and the flow changes its direction, evolving towards a new focal point. The same process carries on repeatedly. It follows that the continuous trajectories are entirely characterised by their successive intersections with the boundaries of regular domains (extending them by continuity, as mentioned previously).\\
This sequence depends essentially on the position of focal points with respect to thresholds. Actually, $\{x\,|\,x_i=\theta_{i}^-(a)\}$ (resp. $\{x\,|\,x_i=\theta_{i}^+(a)\}$) can be crossed if and only if $\phi_i(a) < \theta_{i}^-(a)$ (resp. $\phi_i(a)>\theta_{i}^+(a)$). Then, let us denote $I_{out}^+(a) = \{i\in\NN{n} | \,\phi_i >\theta_{i}^+(a)\}$, and similarly $I_{out}^-(a) = \{i\in\NN{n} | \,\phi_i <\theta_{i}^-(a)\}$. Then, $I_{out}(a)=I_{out}^+(a)\cup I_{out}^- (a)$ is the set of escaping directions of $\mathcal{D}_a$. Also, we call {\it walls} the intersections of threshold hyperplanes with the boundary of a regular domain.\\
When it is unambiguous, we will omit the dependence on $a$ in the sequel. Now, in each direction $i\in I_{out}$ the time at which $x(t)$ encounters the corresponding hyperplane, for $x\in \mathcal{D}_a$, is readily calculated:
\begin{equation}\label{eq-taui}
\tau_i(x)=\frac{-1}{\gamma_i}\ln\left(\frac{\phi_i -\theta_{i}^\pm}{\phi_i -x_i}\right),\qquad i\in I_{out}^\pm.
\end{equation}
Then, % \begin{equation}\label{eq-tau}
$\tau(x)=\min_{i\in I_{out}}\tau_i(x)$,
% \end{equation}
is the {\it exit time} of $\mathcal{D}_a$ for the trajectory with initial condition $x$. Then, we define a \textit{ transition map} ${T }^a: \partial \mathcal{D}_a\rightarrow \partial \mathcal{D}_a$:
\begin{equation}\label{eq-maptrans}
\begin{array}{lcl}
{T }^ax & = & \varphi\left(x,\tau(x)\right)\\
          & = & \phi + \alpha(x) (x-\phi).
\end{array}
\end{equation}
where $\alpha(x) = \exp(-\tau(x)\Gamma)$.\\ 
The map above is defined locally, at a domain $\mathcal{D}_a$. However, under our assumption {\bf (H1)}, any wall can be considered as escaping in one of the two regular domains it bounds, and incoming in the other. Hence, on any point of the interior of a wall, there is no ambiguity on which $a$ to choose in expression~(\ref{eq-maptrans}), and there is a well defined global transition map on the union of walls, denoted $T$. On the boundaries of walls, at intersections between several threshold hyperplanes, the concept of Filippov solution would be required in general \cite{gouzesari}. This problem will either be solved on a case by case basis, or we implicitly restrict our attention to the (full Lebesgue measure) set of trajectories which never intersect more than one threshold hyperplane.\\

To conclude this section let us define the {\it state transition graph} ${\sf TG}$ associated to a system of the form (\ref{eq-genenet}) as the pair $(\mathcal{A},\mathcal{E})$ of nodes and oriented edges, where $\mathcal{A}$ is defined in (\ref{eq-A}) and $(a,b)\in\mathcal{E}\subset \mathcal{A}^2$ if and only if $\partial \mathcal{D}_a\cap \partial \mathcal{D}_b \neq\varnothing$, and there exists a positive Lebesgue measure set of trajectories going from $\mathcal{D}_a$ to $\mathcal{D}_b$. It is not difficult to see that this is equivalent to $b$ being of the form $a\pm \mathrm{e}_i$, with $i\in I_{out}^\pm(a)$ and $\mathrm{e}_i$ a standard basis vector.\\
From now on, it will always be assumed that {\bf (H1)} and {\bf (H2)} hold, at least in some region of state space (or transition graph) on which we focus.\\

\subsection{Illustrative example}
Let us now illustrate the previous notions on a well-known example with two variables, in order to help the reader's intuition. Consider two genes repressing each other's transcription. In the context of piecewise-affine models, this would be described by the system below:
$$
\scalebox{0.9}{\xymatrix{*+[o][F-] {1}  \ar@{|}@/^1pc/[r] & *+[o][F-] {2}  \ar@{|}@/^1pc/[l]}}
\qquad
\left\{\begin{array}{lcl}
\dot x_1 & = & \kappa_1^0+\kappa_1^1{\sf s}^-(x_2,\theta_2^1) - \gamma_1x_1 \\[1mm]
\dot x_2 & = & \kappa_2^0+\kappa_2^1{\sf s}^-(x_1,\theta_1^1) - \gamma_2x_2
\end{array}\right. 
$$
where inhibition is modeled by ${\sf s}^-(x,\theta)$, as already mentioned. A usual notation for the interaction graph uses $\xymatrix @!0 @R=1pc @C=1.5pc
{\ar@{|}[r] &}$ to denote inhibition, and $\xymatrix @!0 @R=1pc @C=1.5pc{\ar[r]&}$ to denote activation.\\
The two constants $\kappa_i^0$ represent the lowest level of production rates of the two species in
interaction. It will be zero in general, but may also be a very low positive constant, in some cases where a
gene needs to be expressed permanently.\\
In the given equation, arbitrary parameters may lead to spurious behaviour, in particular an inhibition which would not drop its target variable below its threshold. To avoid this, it suffices to assume the following conditions on focal points' coordinates:
$$
\frac{\kappa_i^0}{\gamma_i} < \theta_i^1 \quad\text{and}\quad
\frac{\kappa_i^0 +\kappa_i^1}{\gamma_i} > \theta_i^1,\qquad\text{for }i=1,2.
$$
This might be called {\em structural constraints} on parameters. The phase space of this system is schematised on Figure~\ref{fig-explan}.\\
\begin{center}
\begin{figure}
\begin{center}
\scalebox{0.8}{\input{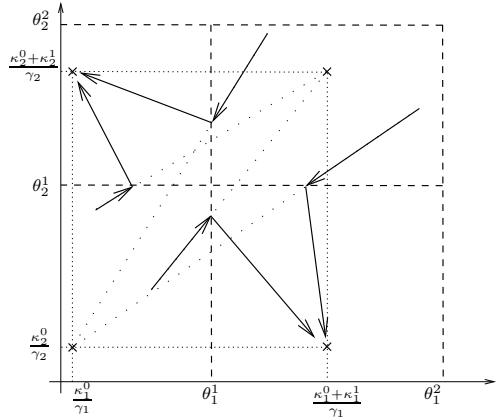}}
\caption{\label{fig-explan} The dashed lines represent threshold hyperplanes, and define a rectangular partition of state space, and dotted lines indicate focal points' coordinates. Arrows represent schematic flow lines, pointed toward these limit points. Note that pieces of trajectories are depicted as straight lines, which is the case when all degradation rates $\gamma_i$ coincide, a fact we never assume in the present study.}
\end{center}
\end{figure}
\end{center}
Then, the transition graph of the system takes the form:
$$
\scalebox{.9}{\xymatrix {
*+[o][F-] {\mathbf{01}}  & 11\ar@*{[|<1pt>]}[l] \ar@*{[|<1pt>]}[d]  \\
00  \ar@*{[|<1pt>]}[r]\ar@*{[|<1pt>]}[u]  & *+[o][F-] {\mathbf{10}}}}
$$
where circled states are those with no successor. It appears in this case that $\mathsf{TG}$ constitutes a reliable abstraction of the system's behaviour. In general, things are not as convenient, and some paths in the transition graph may be spurious. In particular, cyclic paths may correspond to damped oscillations of the original system, but even this cannot be always ascertained without a precise knowledge of the parameter, see section \ref{sec-cyc} for related results. However, one general goal of the present study will be to search for feedback control laws ensuring that given systems are indeed well characterised by their abstraction. Such a property can be deduced from the shape of the abstraction $\mathsf{TG}$ itself, whence the term 'qualitative control'.

\section{Stability and limit cycles}
\label{sec-cyc}
Periodic solutions have soon been a prominent topic of study for systems of the form (\ref{eq-genenet}) \cite{glasspastern2,snoussi,periodsol,edwards,Edw2009}. With the notable exception of \cite{snoussi}, all these studies focused on the special case where $\Gamma$ is a scalar matrix, which greatly simplifies the analysis, since trajectories in each box are then straight lines towards the focal point, as in Figure \ref{fig-explan}. In a recent work \cite{FG06a,ab09,ijss09}, we have shown that the local monotonicity properties of transition maps can be used to prove existence and uniqueness of limit cycles in systems like (\ref{eq-genenet}). In this section we recall without proof some of these results.\\
In the rest of this section we consider a piecewise-affine system such that there exists a sequence  $\mathcal{C}=\{a^0\dots a^{\ell-1}\}$ of regular domains which is a cycle in the transition graph, and study periodic solutions in this sequence. We abbreviate the focal points of these boxes as $\phi^i=\phi(a^i)$. 
%such that $({T}^{a^\ell}\circ{T}^{a^{\ell-1}}\cdots\circ{T}^{a^1}) (W)\cap W\ne \varnothing$.  Then, fixed points of such an iterate are equivalent to periodic trajectories in the original system~(\ref{eq-genenet}). 
Let us now define a property of these focal points: we say that the points $\phi^i$ are {\it aligned} if
\begin{equation}\label{eq-align}
\forall i\in\Nm{\ell},\,\exists! j\in\NN{n},\quad\phi^{i+1}_j-\phi^i_j\neq 0,
\end{equation}
where $\phi^{\ell}$ and $\phi^0$ are identified.\\
Since $\mathcal{C}$ is supposed to be a cycle in $\mathsf{TG}$, for each pair $(a^i,a^{i+1})$ of successive boxes there must be at least one coordinate at which their focal points differ, namely the only $s_i\in I_{out}(a^i)$ such that $a^{i+1}=a^i\pm\mathrm{e}_{s_i}$. We keep on denoting $s_i$ this {\it switching} coordinate in the following.
 Hence condition (\ref{eq-align}) means that $s_i$ is the only coordinate in which $\phi^i$ and $\phi^{i+1}$ differ. This implies in particular that $a^{i+1}$ is the only successor of $a^i$, i.e. there is no edge in $\mathsf{TG}$ from $\mathcal{C}$ to $\mathcal{A}\setminus\mathcal{C}$. It might seem  intuitive in this case that all orbits in $\cal C$ converge either to a unique limit cycle, or to a point at the intersection of all crossed thresholds. However, this fact has only been proved for uniform decay rates (i.e. $\Gamma$ scalar), \cite{glasspastern2}, and its validity with distinct decay remains an open question. \\
The condition (\ref{eq-align}) is of purely geometric nature. However, it can be shown that it holds necessarily when the interaction graph has degree one or less, see \cite{ijss09} for more details.\\ 
If $\{s_i\}_{0\leqslant i<\ell}=\NN{n}$, i.e. all variables are switching along $\cal C$, then the intersection of all walls between boxes in $\cal C$ is either a single point, which we denote $\theta^{\mathcal{C}}$, or it is empty. The latter holds when two distinct thresholds are crossed in at least one direction. When defined,  $\theta^{\mathcal{C}}$ is a fixed point for any continuous extension of the flow in $\cal C$, see \cite{ijss09}.\\
Let us now rephrase the main result from \cite{ijss09}.
\begin{theorem}\label{thm-main}
 Let ${\cal C}=\{a^0,a^1\cdots a^{\ell-1}\}$ denote a sequence of regular domains which is periodically visited by the flow, and whose focal points satisfy condition (\ref{eq-align}).
 %, i.e. they are aligned.
 Suppose also that all variables are switching at least once.\\
Let $W$ denote the wall $\partial \mathcal{D}_{a^0}\cap\partial \mathcal{D}_{a^{1}}$, and consider the first return map $\mathbf{T}:W\to W$ defined as the composite of local transition maps along $\mathcal{C}$.\\
$A)$ If a single threshold is crossed in each direction, let $\lambda =\rho(D\mathbf{T}(\theta^{\mathcal{C}}))$, the spectral radius of the differential $D\mathbf{T}(\theta^{\mathcal{C}})$.
Then, the following alternative holds:\\
%\begin{description} 
%\item[i)] 
\hspace*{3mm}{\bf i)} if $\lambda\leqslant 1$, then $\forall x\in W$, $\mathbf{T}^n x\to \theta^{\mathcal{C}}$ when $n\to\infty$. \\
%\item[ii)] 
\hspace*{3mm}{\bf ii)} if $\lambda>1$ then there exists a unique fixed point different from  $\theta^{\mathcal{C}}$, say $q=\mathbf{T} q$. Moreover, for every $x\in W\setminus\{\theta^{\mathcal{C}}\}$, $\mathbf{T}^n x\to q$ as $n\to\infty$. \\
%\end{description} 
$B)$ If there are two distinct crossed thresholds in at least one direction, then the conclusion of {\bf ii)} holds.
\end{theorem}

 In \cite{FG06a,ab09} we have resolved the alternative above for a particular class of systems:
\begin{theorem}\label{thm-negloop}
Consider a negative feedback loop system of the form 
$$
\dot{x}_i= \kappa_i^0+{\sf s}^{\varepsilon_i}(x_{i-1},\theta_{i-1})-\gamma_ix_i,\qquad \varepsilon_i\in\{-,+\}\quad i\in\NN{n},
$$
with subscripts understood modulo $n$, and an odd number of negative $\varepsilon_i$. It can be shown that there exists a cycle $\cal C$ in $\mathsf{TG}$ whose focal points satisfy (\ref{eq-align}).\\
Then, in Theorem \ref{thm-main}, A.{\bf i)} holds in dimension $n=2$, and A.{\bf ii)} holds for all $n\geqslant 3$.
\end{theorem}

\section{Piecewise Control}
\label{sec-con}
Feedback regulation is naturally present in many biological systems, as the widespread appelation 'regulatory network' suggests. Hence, it seems appropriate to take advantage of the important body of work developed in feedback control theory for decades, in order to study gene regulatory networks and related systems \cite{iglesiasingalls,sontag05}.\\
In particular, the recent advent of so called {\it synthetic biology}~\cite{msb06,pnas-22-04}
%\cite{msb06,letnat1,letnat2,pnas-22-04,shimizu02,stricker08}
, has led to a situation where gene regulatory processes are not only studied, but designed to perform certain functions. Hence, autonomous systems of the form (\ref{eq-genenet}) should to be extended, so as to include possible input variables. In~\cite{FG08}, we have presented such an extension, where both production and decay terms have some additional argument $u\in\mathbb{R}^p$, of which they were affine functions. In this context, we defined a class of qualitative control problems, and showed that were equivalent to some linear programming problems.\\
As in our previous work, we consider qualitative feedback laws, in the sense that they depend only on the box containing the state vector, rather than its exact value. This choice is motivated by robustness purposes. More pragmatically, it is also due to the fact that recent techniques allow for the observation of qualitative characteristics of biological systems, for instance by live imaging, using confocal microscopy, of GFP marker lines, where the measured state is closer to an ON/OFF signal than to a real number.\\
Recent experimental techniques allow furthermore for the reversible induction of specific genes at a chosen instant, for instance using promoters inducible by ethanol~\cite{deveaux03}, %miRNA\cite{schwab06}, silencing by RNA \cite{chen03}
 or light \cite{shimizu02}, to name only two. Also, degradation rates may be modified, either directly by introducing a drug \cite{wyke06}, or via a designed genetic circuit \cite{grilly07}, which might be induced using previously mentioned techniques.\\
 
To simplify the presentation, we focus in this paper on the particular case where decay rates can be linearly controlled by a scalar and bounded input $u$. For each $ i\in\NN{n}$, let us denote this as:
\begin{equation}\label{eq-genenetu}
\frac{dx_i}{dt} = \kappa_i(x) - (\gamma_i^1(x)u+\gamma_i^0(x)) x,\quad u\in[0,U]\subset\mathbb{R}_+,
\end{equation}
where $\gamma_i^0$ and $\gamma_i^1$ are piecewise constant functions assumed to satisfy $\gamma_i^0>0$ and $\gamma_i^1>- \frac{\gamma_i^0}{U}$, in any box. This ensures that decay rates are positive, but yet can be decreasing functions of $u$ (for $\gamma_i^1<0$).\\
Now, a feedback law depending only on the qualitative state of the system is simply a expressed as the composite of a map $\bigcup_a\mathcal{D}_a\to\mathcal{A}$ indicating the box of the current state, with a function $u:\mathcal{A}\to[0,U]$ which represents the control law itself. In other words, in each box a constant input value is chosen. For a fixed law of this form, it is clear that the dynamics of (\ref{eq-genenetu}) is entirely determined, and in particular we denote its transition graph by $\mathsf{TG}(u)$.\\
Let us now recall our definition of control problem.\\[3mm]
{\bf Global Control Problem}: Let $\mathsf{TG}^\star =(\mathcal{A},\mathcal{E}^\star)$ be a transition graph. Find a feedback law $u:{\cal
A}\to[0,U]$ such that $\mathsf{TG}(u)=\mathsf{TG}^\star$.\\[3mm]
Clearly, $\mathcal{E}^\star$ cannot be arbitrary in $\mathcal{A}^2$, and must in particular contain only arrows of the form $(a,a\pm\mathrm{e}_i)$. Now in the present, restricted, context the equivalent linear programming problem described in \cite{FG08} is very simple. %Actually, since $\mathcal{A}$ is finite, $u$ can be seen as a point in a euclidean space $\mathbb{R}^{\#\mathcal{A}}$ (where $\#$ denotes cardinality).\\
For each $a\in\mathcal{A}$, the control problem above requires that the focal point $\phi(a,u(a))$ belongs to a certain union of boxes, i.e. its coordinates must satisfy inequalities of the form $\theta_i^{j^-(a)}< \kappa_i(a)/ (\gamma_i^1(a)u(a)+\gamma_i^0(a)) <\theta_i^{j^+(a)}$, or equivalently
\begin{equation}\label{eq-u}
\frac{ \kappa_i(a)- \gamma_i^0(a)\theta_i^{j^+(a)}}{\gamma_i^1(a)\theta_i^{j^+(a)}}< u(a) <\frac{ \kappa_i(a)- \gamma_i^0(a)\theta_i^{j^-(a)}}{\gamma_i^1(a)\theta_i^{j^-(a)}}
\end{equation}
if $\gamma_i^1(a)>0$, and in reverse order otherwise. Hence, the solution set of the control problem is just the Cartesian product of all intervals of the form (\ref{eq-u}), when $a$ varies in $\mathcal{A}$. It is thus identical to a rectangle in $\mathbb{R}^{\#\mathcal{A}}$ (where $\#$ denotes cardinality), which is of full dimension if and only if the problem admits a solution.\\
Thanks to the explicit description (\ref{eq-u}),  this set can be computed with a complexity which is linear in $\#\mathcal{A}$. The latter grows exponentially with the dimension of the system, but in practice, one will face problems where $\mathcal{E}$ and $\mathcal{E}^\star$ only differ on a subset of initial vertices, say $\mathcal{A}^\star$, and then the actual complexity will be of order $\#\mathcal{A}^\star$.\\

In addition to this type of control, we introduce in this note some first hints toward dynamic feedback control, where instead of a direct feedback $u$, one uses some additional variable (here a single one), evolving in time according to a system of the form (\ref{eq-genenet}), and coupled to the initial system. This is suggested by the rectangular form of admissible inputs found in (\ref{eq-u}): instead of fixing an arbitrary value in a rectangle of an external input space, one increases the state space dimension, which has the effect of adding new boxes to the system. The dynamics of the supplementary variables is then defined by analogy with the direct feedback case: when the initial variables are in a box $a$, this makes additional variables tend to a box of the form (\ref{eq-u}).\\
This raises a number of questions, in large part due to the fact that instead of applying an input $u(a)$ instantaneously when entering box $\mathcal{D}_a$, the feedback now tends toward some value, which takes some time. Instead of fully developing a general theory, we thus have to chosen to illustrate it on a simple example, in section \ref{sec-ex1}.

\section{Examples}
\label{sec-ex}
We now illustrate with examples how it is possible to combine results of the two previous sections, and compute qualitative feedback laws ensuring (or precluding) the existence and uniqueness of oscillatory behaviour of a system of the form (\ref{eq-genenetu}).

\subsection{Example 1: disappearance of a limit cycle}\label{sec-ex1}

Consider the following two dimensional system:
\begin{equation}\label{eq-ex1}
\left\{
\begin{array}{l}
\dot x_1(t)  =  K_1{\sf s}^-(x_2)   -(\gamma_1^1 u+\gamma_1^0) x_1\\[2mm]
\dot x_2(t)  =  K_2[ {\sf s}^+(x_1,\theta_1^1){\sf s}^+(x_2) + {\sf s}^+(x_1,\theta_1^2){\sf s}^-(x_2)]  -\gamma_2^0 x_2
\end{array}
\right.
\end{equation}
where $x_2$ has a unique threshold, and ${\sf s}^\pm(x_2)={\sf s}^\pm(x_2,\theta_2^1)$. We assume moreover that the following inequalities stand:
\begin{equation}\label{eq-structconstex1}
\gamma_1^1>0,\quad K_1>\gamma_1^0\theta_1^2,\quad K_2>\gamma_2^0\theta_2^1,
\end{equation}
so that the first decay rate increases with $u$. Also, the interactions are {\it functional}: an activation of a variable leads to the corresponding focal point coordinate being above a variable's threshold (chosen as the highest one for $x_1$, since otherwise $\theta_1^2$ cannot be crossed from below). Remark that in this system, $x_2$ violates {\bf (H1)}. However, it will appear soon that this autoregulation is only effective at a single wall, which is unstable, and thus can be ignored safely.\\
This system corresponds to a negative feedback loop, where $x_2$ is moreover able to modulate its activation by $x_1$: when $x_2$ is above its threshold, the interaction is more efficient, since it is active at a lower threshold $\theta_1^1<\theta_1^2$. Biologically, this may happen if the proteins coded by $x_1$ and $x_2$ form a dimer, which activates $x_2$ more efficiently than $x_1$ protein alone. This is reminiscent of the {\it mixed feedback loop}, a very widespread module able to display various behaviours~\cite{mfl05}. It might be depicted by this graph
%$$\SelectTips{eu}{10}
%\xymatrix{
%E  & A \ar[r] ^-{}="b" & B \ar@/^1pc/ "1,1";"b"} $$
$$
\SelectTips{eu}{11}
\xymatrix @C=1.5cm{
   *+[o][F-] {1}  \ar@*{[|<.5pt>]}@/^.8pc/[r] _-{}="name" &
   *+[o][F-] {2}  \ar@*{[|<.5pt>]}@/^.8pc/@{-|}[l]  \ar@/_1pc/ "1,2";"name"
 }
$$ 

As seen in the equations, the scalar input is assumed to affect the first decay rate, but not the second (i.e. $\gamma_2^1=0$). Now, one readily computes the focal points of all boxes:
\begin{equation}\label{tab-focex1}\begin{array}{|c|c|c|c|c|c|}
\hline 
00 & 01& 10 & 11 & 20 & 21  \\
\hline 
\phi_1 & 0 &\phi_1 &0 &\phi_1 &0  \\
0 &0 &0 &\phi_2 & \phi_2 & \phi_2 \\
\hline 
\end{array}\end{equation}
where $\phi_1$ is an abbreviation for $K_1/(\gamma_1^1 u+\gamma_1^0)$, and $\phi_2$ for $ K_2/\gamma_2^0$. Under the constraints (\ref{eq-structconstex1}), this readily leads to the transition graph in absence of input (i.e. $u=0$ in all boxes):
$$
\mathsf{TG}(0)=\xymatrix {01  \ar@*{[|<1.5pt>]}[d] & 11\ar@*{[|<1.5pt>]}[l] \ar@{.}@*{[|<.5pt>]}[d] & 21\ar@*{[|<1.5pt>]}[l]\\
00  \ar@*{[|<1.5pt>]}[r] & 10\ar@*{[|<1.5pt>]}[r] & 20\ar@*{[|<1.5pt>]}[u]  }
$$
The dotted line represents an unstable wall, for which Filippov theory would be required for full rigour. However, this wall is not reachable, and we ignore it afterwards.\\
Now, since this graph has a cycle, the two thresholds $\theta_1^{1,2}$ are crossed, and (\ref{tab-focex1}) is easily seen to imply condition (\ref{eq-align}), conclusion $B)$ of Theorem \ref{thm-main} applies : there is a unique stable limit cycle.\\
Now, in accordance with the section's title, let us look for a $u$ leading to: $$
\mathsf{TG}^\star=\xymatrix {01  \ar@*{[|<1.5pt>]}[d] & 11\ar@*{[|<1.5pt>]}[l] \ar@{.}@*{[|<.5pt>]}[d] & 21\ar@*{[|<1.5pt>]}[l]\\
00  \ar@*{[|<1.5pt>]}[r] & *+[o][F-] {\mathbf{10}} & 20\ar@*{[|<1.5pt>]}[u] \ar@*{[|<1.5pt>]}[l] }$$
Clearly from $\mathsf{TG}^\star$, the box $\mathcal{D}_{10}$ attracts trajectories from all other boxes, and contains its own focal point, which is thus a globally asymptotically stable equilibrium. The only states whose successors differ in $\mathsf{TG}(0)$, and $\mathsf{TG}^\star$ are $10$ and $20$, hence we assume $u(a)=0$ for all other $a\in\mathcal{A}$, or $\mathcal{A}^\star=\{10,20\}$ to recall the notations of previous section. Then, Eq. (\ref{eq-u}) with $\theta_1^{j^-(a)}=\theta_1^1$ and $\theta_1^{j^+(a)}=\theta_1^2$ gives:
\begin{equation}\label{eq-solex1}
\frac{ K_1- \gamma_1^0\theta_1^{2}}{\gamma_1^1\theta_1^{2}}< u(a) <\frac{ K_1- \gamma_1^0\theta_1^{1}}{\gamma_1^1\theta_1^{1}}
\end{equation}
for both $a\in\mathcal{A}^\star$. This defines a nonempty interval by $\theta_1^1<\theta_1^2$, hence the Control Problem of previous section can be solve under constraints (\ref{eq-structconstex1}). An illustration on a numerical example is shown Figure~\ref{fig-ex1-2D}.\\
\begin{figure}
\centering
\includegraphics[height=6cm]{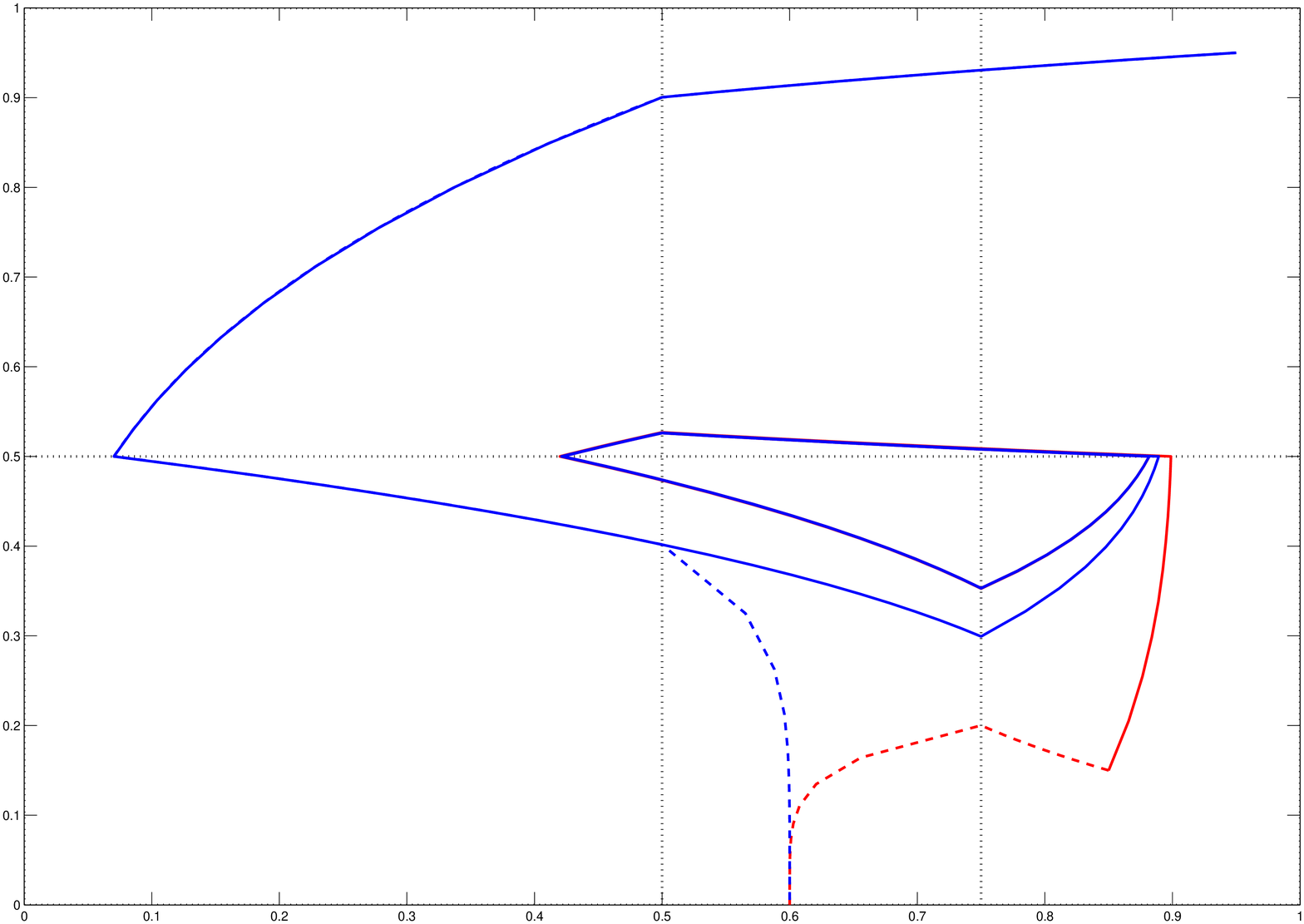}%fig_ex1_oneCI.eps,height=7cm}
\caption{\label{fig-ex1-2D}Dashed lines: with feedback control. Plain lines: without.  Two initial conditions, $(0.95,0.95)$ in box 21 (blue curves) and $(0.85,0.15)$ in 20 (red curves). The controlled and autonomous trajectories only diverge in box 10, 20  where the feedback is active. See parameters in Appendix \ref{sec-param-1}}
\end{figure}

Now, let us focus on the question of realising an extended network which solves the same problem, by adding a variable to system (\ref{eq-ex1}). In other words, one now seeks to impose the dynamics described by $\mathsf{TG}^\star$ using dynamic feedback. Biologically, this amounts to designing a genetic construct whose promoter depends transcriptionaly on $x_1$ and $x_2$, and increases the degradation rate of $x_1$. Let us denote by $y$ the expression level of this additional gene. The most obvious version of such an extended system arises by increasing $y$ production rate exactly at boxes in $\mathcal{A}^\star$:
\begin{equation}\label{eq-ex1bis}
\left\{
\begin{array}{l}
\dot x_1(t)  =  K_1{\sf s}^-(x_2)   -(\gamma_1^1\upsilon\,\mathsf{s}^+(y)+\gamma_1^0) x_1\\[2mm]
\dot x_2(t)  =  K_2 [ {\sf s}^+(x_1,\theta_1^1){\sf s}^+(x_2) + {\sf s}^+(x_1,\theta_1^2){\sf s}^-(x_2)]  -\gamma_2^0 x_2\\[2mm]
\dot y(t)  =  {\sf s}^+(x_1,\theta_1^1){\sf s}^-(x_2) -\gamma_y y
\end{array}
\right.
\end{equation}
$\upsilon$ a constant in the interval (\ref{eq-solex1}), so that forcing $s^+(y)=1$ would lead us back to a static feedback solution. This use of a single constant $\upsilon$ is possible in this particular example because constraints (\ref{eq-solex1}) are identical for the two boxes in $\mathcal{A}^\star$, but it should be noted that in general several constants might be required.\\
We consider without loss of generality that  $y\in[0,1/\gamma_y]$, since higher values of $y$ tend to $1/\gamma_y$ or $0$. Also, $\mathsf{s}^+(y)$ is defined with respect to a threshold $\theta_y\in(0,1)$. We also assume $\theta_y\gamma_y<1$, ensuring that $y$ may cross its threshold when activated.\\
Now, (\ref{eq-ex1bis}) defines an autonomous systems of the form (\ref{eq-genenet}), whose transition graph has indeed a fixed point $101$:
\begin{equation}\label{eq-tgy}\scalebox{.9}{
\xymatrix @ -.5cm {&011 \ar@*{[|<.5pt>]}'[d][dd]\ar@*{[|<1.5pt>]}[dl] && 111\ar@*{[|<1.5pt>]}[ll]  \ar@*{[|<.5pt>]}'[d][dd]  \ar@{.}@*{[|<.5pt>]}[dl]   && 211\ar@*{[|<.5pt>]}[dd] \ar@*{[|<1.5pt>]}[ll]\\
001 \ar@*{[|<.5pt>]}[dd]   \ar@*{[|<1.5pt>]}[rr] &&  *+[o][F-] {\bf 101}  && 201\ar@*{[|<1.5pt>]}[ll]  \ar@*{[|<1.5pt>]}[ur] &\\
&010  \ar@*{[|<1.5pt>]}[dl] && 110\ar@*{[|<1.5pt>]}'[l][ll]  \ar@{.}@*{[|<.5pt>]}[dl]   && 210  \ar@*{[|<1.5pt>]}'[l][ll]   \\
000  \ar@*{[|<1.5pt>]}[rr] && 100 \ar@*{[|<1.5pt>]}[rr]\ar@*{[|<.5pt>]}[uu] && 200\ar@*{[|<.5pt>]}[uu]\ar@*{[|<1.5pt>]}[ur]& 
}  }
\end{equation}
This fixed point corresponds the fixed point $10$ of $\mathsf{TG}^\star$: in fact, the upper part of the graph above, where $\mathsf{s}^+(y)=1$ is exactly $\mathsf{TG}^\star$. However, it is not invariant, and some trajectories can escape to $\mathsf{s}^+(y)=0$, where we see $\mathsf{TG}(0)$, and thus the possibility of periodic solutions. Besides, there are other cycles in this graph.\\
% and possibly cycle through (for instance) the sequence $\{000,100,200,201,211,111,011,001\}$
Unlike static feedback control -- and more realistically -- the effect of $y$ on $\gamma_1$ takes some positive time, explaining why the situation is not a direct translation of previous case. We will now show that under additional constraints of the parameters governing $y$'s dynamics, it is possible to guarantee that $\mathcal{D}_{101}$ contains a globally asymptotically stable equilibrium. To achieve this, let us rephrase a lemma, proved as Lemma 1 in \cite{farcot06}:
\begin{lemma}\label{lem}
For any box, there is at most one pair of parallel walls successively crossed by solution trajectories of a system of the form~(\ref{eq-genenet}).
\end{lemma}
In other words, there is at most one direction $i$ such that opposite walls, of the form $x_i=\theta_i^{-}$ and $x_i=\theta_i^{+}$, are crossed. Moreover, such an $i$ is characterised, see \cite{farcot06}, by the condition
\begin{equation}\label{eq-lem1}
\forall j\neq i, \quad \tau_i(\theta_i^{-})<\tau_j(\theta_j^{-}),
\end{equation}
under the assumption $I_{out}=I_{out}^+$ (which simplifies the description without loss of generality), i.e. all exiting walls occur at higher threshold values, of the form $\theta_i^{+}$, which is thus the threshold involved in the definition of $\tau_i$, Eq. (\ref{eq-taui}). This allows us to prove the following result:
\begin{prop}
\label{prop-ex1}
Suppose that the parameters of (\ref{eq-ex1bis}) satisfy, denoting $\phi_1=\frac{K_1}{\gamma_1^0}$:%+\gamma_1^1\upsilon}$:
$$
\left(1-\gamma_y\theta_y\right)^{\frac{1}{\gamma_y}} > \left(\frac{\phi_1-\theta_1^2}{\phi_1-\theta_1^1}\right)^{\frac{1}{\gamma_1^0}}%+\gamma_1^1\upsilon}}
$$
Then there the steady state in box $\mathcal{D}_{101}$ attracts the whole state space of system (\ref{eq-ex1bis}).
\end{prop}
\begin{proof}
Since each box is either containing an asymptotic steady state, or has all its trajectories escaping it toward a focal point, all limit set must be contained in a strongly connected component of the transition graph, i.e. a collection of cyclic paths sharing some vertices.
A visual inspection of the transition graph displayed in (\ref{eq-tgy}) shows that any cyclic path in the transition graph $\mathsf{TG}$ must visit the state $100$. This state has only two successors: $200$ and $101$, the fixed state. Hence it has two exit walls, which we denote by: $W_1^+=\{\theta_1^2\}\times (\theta_2^0,\theta_2^1)\times (0,\theta_y)$ and $W_y^+=(\theta_1^1,\theta_1^2)\times (\theta_2^0,\theta_2^1)\times\{\theta_y\}$. All other walls are incoming. Denoting them by obvious analogy with the two exiting walls, let us consider each of them. First, both walls $W_2^\pm$ are repelling: this has already been said for $W_2^+$ when discussing auto-regulatory terms in (\ref{eq-ex1}). For $W_2^-$, this follows from the fact that $\mathcal{D}_{100}$ contains only trajectories escaping in finite time, and can be extended to this wall by continuity. Moreover, it follows from $\mathsf{TG}$ that any trajectory escaping $W_2^\pm$ either reaches $\mathcal{D}_{101}$ (where the fixed point lies), or enters $\mathcal{D}_{100}$ again via the wall $W_1^-$. Thus, any trajectory which does not enter $\mathcal{D}_{101}$ must cross the pair $W_1^\pm$ in succession. Now, from Lemma~\ref{lem}, among the two pairs of walls $W_1^\pm$ , $W_y^\pm$, only one can be crossed in succession by trajectories. Moreover, the inequality in the statement is the exact translation of the condition (\ref{eq-lem1}), in the case where $W_y^\pm$ is the crossed pair of walls, precluding any attractor but the known fixed point, $\phi(101)$.
\end{proof}

Some elementary calculus shows that the left-hand side in the inequality of proposition~\ref{prop-ex1} is a increasing function of $\gamma_y$ when $\gamma_y\in(0,1/\theta_y)$, as assumed previously. Thus, this inequality is equivalent to requiring a lower bound to $\gamma_y$, eventhough this bound does not have a simple explicit form. \\
This fact can be given an intuitive explanation: $\gamma_y$ is inversely proportional to the characteristic time of the variable $y$, in each box. Hence, proposition~\ref{prop-ex1} means that the dynamics of $y$ must be fast enough in order to retrieve the behaviour of the static feedback control, which corresponds to the limit of an instantaneous feedback. See Figure~\ref{fig-ex1-3D} for a numerical example.
\begin{figure}
\hspace*{-.5cm}\includegraphics[height=7.5cm,width=9.5cm]{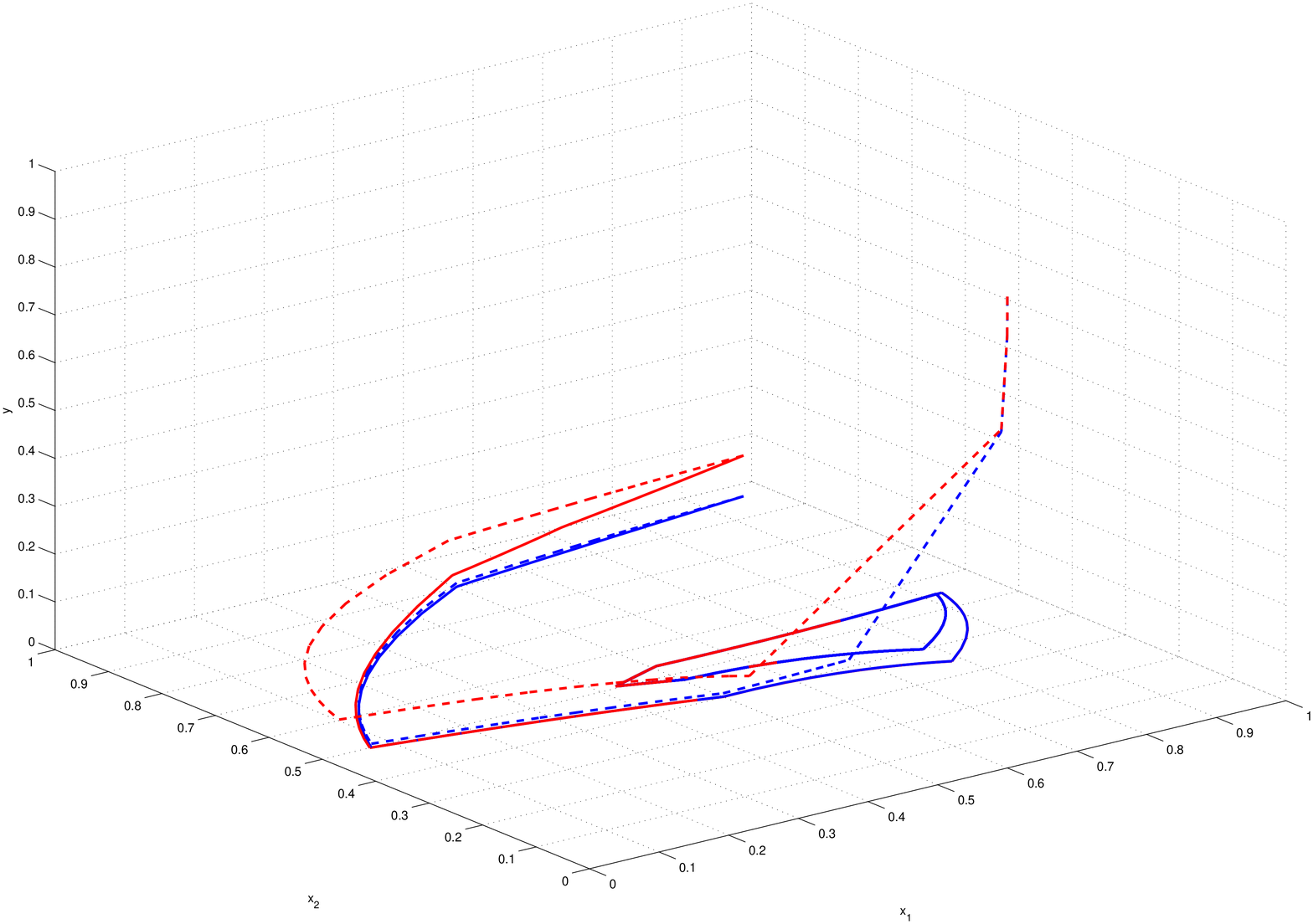}
\caption{\label{fig-ex1-3D}Dashed line: inequality of proposition \ref{prop-ex1} satisfied. Plain line: inequality violated. Two common initial conditions, $(x_1,x_2,y)=(0.95;0.95,0.1)$ (blue) and $(0.95,0.95,0.95)$ (red). The value of $y$ has been divided by $10$ to keep all variables in $[0,1]$. In both cases a limit cycle is controlled into an equilibrium point. See parameters in Appendix \ref{sec-param-2}}
\end{figure}

The results of this section can be summarised as
\begin{prop}\label{mainprop-ex1}
A system of the form (\ref{eq-ex1}), with structural constraints (\ref{eq-structconstex1}), has a unique, stable and globally attractive limit cycle in absence of input, i.e. $u=0$.\\
There exists a control law ensuring a unique, stable and globally attractive equilibrium point. This control can be achieved in two ways:\\
\hspace*{3mm}$\bullet)$ Using a scalar piecewise constant feedback $u$, such that $u(a)$ satisfies (\ref{eq-solex1}) for $a\in\{10,20\}$.\\
\hspace*{3mm}$\bullet)$ Using dynamic feedback with a single additional variable $y$, as in (\ref{eq-ex1bis}), whose decay rate satisfies the condition in proposition \ref{prop-ex1}, and with $\upsilon$ a solution of (\ref{eq-solex1}).
\end{prop}

\subsection{Example 2 : birth of a limit cycle}
Let us now consider the following system
\begin{equation}\label{eq-ex2}
\left\{
\begin{array}{l}
\dot x_1(t)  =  K_1\: {\sf s}^+(x_2)- (\gamma_1^1u+\gamma_1^0)\: x_1\\
\dot x_2(t)  =  K_2^3\: {\sf s}^-(x_3)+ K_2^1\:{\sf s}^-(x_1,\theta_1^2) - (\gamma_2^1u+\gamma_2^0)\: x_2\\
\dot x_3(t)  =  K_3\:{\sf s}^+(x_1,\theta_1^1) - (\gamma_3^1u+\gamma_3^0)\: x_3\\
\end{array}\right.
\end{equation}
where ${\sf s}^+(x_i)$ abbreviates ${\sf s}^+(x_i,\theta_i^1)$ for $i=2,3$. We assume the following constraints  to be satisfied
\begin{equation}\label{eq-structconstex2}
\left\{\begin{array}{l}
\gamma_i^1>0,\;\; i=1,2,3,\quad K_1\;>\theta_1^2\:\gamma_1^0\;> \theta_1^1\:\gamma_1^0\\[1mm]
 K_3>\theta_3\gamma_3^0,\quad K_2^i>\theta_2\gamma_2^0,\; i=1,3.
\end{array}\right.
\end{equation}
This system is a particular case of two combined negative feedback loops, of the form:
$$\scalebox{.7}{
\SelectTips{eu}{11}\xymatrix @ -.5pc {*+[o][F-] {3}  \ar@<1mm>@*{[|<.5pt>]}@/^.8pc/@{-|}[rr]   && *+[o][F-] {2} \ar@<1mm>@*{[|<.5pt>]}@/^.8pc/[dl] \\
 & *+[o][F-] {1}\ar@<1mm>@*{[|<.5pt>]}@/^.8pc/[ul] \ar@<1mm>@*{[|<.5pt>]}@/^.8pc/@{-|}[ur]&}
}$$
Since the behaviour of a single loop is well characterised by theorem~\ref{thm-negloop}, it can be considered as one of the simplest systems whose behaviour might be worth investigating.\\
Computing the focal points of all boxes, with the abbreviations $\phi_i=K_i/(\gamma_i^1u+\gamma_i^0)$ (with additional superscripts to $\phi_i$ and $K_i$ for $i=2$) and  $\phi_2^+=\phi_2^1+\phi_2^3$, leads to the following table:\\[2mm]
%\begin{center}
\scalebox{.9}{$%\begin{equation}\label{tab-focex2}
\begin{array}{|c|c|c|c|c|c|c|c|c|c|c|c|}
\hline
000 & 100 & 200& 010& 110& 210& 001& 101& 201& 011& 111& 211 \\
\hline 
0    &0    &0    &\phi_1&\phi_1&\phi_1&0&0      &0    &\phi_1&\phi_1&\phi_1 \\
\phi_2^+&\phi_2^+&\phi_2^3    &\phi_2^+&\phi_2^+&\phi_2^3    &\phi_2^1&\phi_2^1 &0 &\phi_2^1 &\phi_2^1 &0    \\
0    &\phi_3&\phi_3&0    &\phi_3&\phi_3&0&\phi_3&\phi_3&0     &\phi_3&\phi_3\\
\hline
\end{array}
%\end{equation}
$}

Under the indicated parameter constraints, the following transition graph is easily deduced:
\begin{equation}
\mathsf{TG}(0)=\scalebox{.8}{\xymatrix @ -.5cm {&011  \ar@*{[|<1.5pt>]}[rr]\ar@*{[|<1.5pt>]}'[d][dd]&& 111  \ar@*{[|<1.5pt>]}[rr]   && 211\ar@*{[|<1.5pt>]}[dl]\\
001  \ar@*{[|<1.5pt>]}[dd] \ar@*{[|<1.5pt>]}[ur]&& 101\ar@*{[|<1.5pt>]}[ur]\ar@*{[|<1.5pt>]}[ll]  && 201\ar@*{[|<1.5pt>]}[ll]  &\\
&010  \ar@*{[|<1.5pt>]}'[r][rr] && 110\ar@*{[|<1.5pt>]}'[u][uu] \ar@*{[|<1.5pt>]}'[r][rr] && 210  \ar@*{[|<1.5pt>]}[uu]  \ar@*{[|<1.5pt>]}[dl] \\
000  \ar@*{[|<1.5pt>]}[ur] && 100 \ar@*{[|<1.5pt>]}[ll]\ar@*{[|<1.5pt>]}[uu]\ar@*{[|<1.5pt>]}[ur]   && 200\ar@*{[|<1.5pt>]}[uu]\ar@*{[|<1.5pt>]}[ll]& 
}  }
\end{equation}
The region with bold arrows -- i.e. the whole graph in this case -- is invariant, and we see that depending on the parameter values, the actual solutions of (\ref{eq-ex2}) may have various behaviours: to each periodic path in $\mathsf{TG}(0)$, a stable periodic orbit may possibly correspond, and there is an infinity of such paths. Some examples have already been provided of such situations, where periodic paths of arbitrary length can be realised as stable limit cycles, by suitable choice of parameters~\cite{gedeon03}. Although it does not present a fixed box, it may also have a stable equilibrium, limit of damped oscillations, as will be illustrated soon.\\
In order to guarantee that the system oscillates, we fix the following objective:
\begin{equation}
\mathsf{TG}^\star=\scalebox{.8}{\xymatrix @ -.5cm {&011  \ar@*{[|<.5pt>]}[rr]\ar@*{[|<.5pt>]}'[d][dd]\ar@*{[|<.5pt>]}[dl] && 111\ar@*{[|<1.5pt>]}[dl]  && 211\ar@*{[|<.5pt>]}[dl] \ar@*{[|<.5pt>]}[ll]  \\
001  \ar@*{[|<1.5pt>]}[dd] && 101\ar@*{[|<1.5pt>]}[ll]  && 201\ar@*{[|<.5pt>]}[ll]  &\\
&010  \ar@*{[|<1.5pt>]}'[r][rr] && 110\ar@*{[|<1.5pt>]}'[u][uu]  && 210 \ar@*{[|<.5pt>]}'[l][ll] \ar@*{[|<.5pt>]}[uu]  \\
000  \ar@*{[|<1.5pt>]}[ur] && 100 \ar@*{[|<.5pt>]}[ll]\ar@*{[|<.5pt>]}[uu]\ar@*{[|<.5pt>]}[ur]   && 200\ar@*{[|<.5pt>]}[uu]\ar@*{[|<.5pt>]}[ll] \ar@*{[|<.5pt>]}[ur]& 
}  }
\end{equation}
We now see that the invariant region in bold is a cycle with no escaping edge. Furthermore, it lies in the region where $\mathsf{s}^-(x_1,\theta_1^2)=1$, and it appears from (\ref{eq-ex2}) that the system in this region is a negative feedback loop involving the three variables. Hence, from theorem~\ref{thm-negloop}, we can conclude that there exists a unique stable limit cycle, which is globally attractive, as deduced from $\mathsf{TG}^\star$. Yet, it remains to state the inequalities defining this graph. They follow from the inversion of arrows in contact with some $a\in\mathcal{A}^\star=\{110,210,111,211, 001,101,011\}$, which leads to
$$
\left\{\begin{array}{lclcl}
\theta_1^2\:(\gamma_1^0+u\gamma_1^1) & > & K_1 & > & \theta_1^1\:(\gamma_1^0+u\gamma_1^1)\\
  K_2^1+K_2^3&>& K_2^3 & > & \theta_2^1\:(\gamma_2^0+u\gamma_2^1) \\
&& K_2^1 & < & \theta_2^1\:(\gamma_2^0+u\gamma_2^1) \\
&& K_3 & > & \theta_3^1\:(\gamma_3^0+u\gamma_3^1) 
\end{array}\right.
$$
This system, following (\ref{eq-u}), can be reduced to:
\begin{equation}\label{eq-solex2}
\begin{array}{ll}
\displaystyle\max \left\{\frac{ K_1- \gamma_1^0\theta_1^{2}}{\gamma_1^1\theta_1^{2}},\frac{ K_2^1- \gamma_2^0\theta_2^{1}}{\gamma_2^1\theta_2^{1}}\right\}< u(a) &\\
\displaystyle< \min \left\{\frac{ K_2^3- \gamma_2^0\theta_2^{1}}{\gamma_2^1\theta_2^{1}}, \frac{ K_1- \gamma_1^0\theta_1^{1}}{\gamma_3^1\theta_3^{1}},\frac{ K_3- \gamma_3^0\theta_3^{1}}{\gamma_3^1\theta_3^{1}}\right\}&
\end{array}\end{equation}
for $a\in\mathcal{A}^\star$. The problem is thus reduced to the satisfiability of the inequality between the two extreme terms above. This fact holds for some parameter values satisfying constraints (\ref{eq-structconstex2}), and the results of this section can be summarised as
\begin{prop}\label{mainprop-ex2}
A system of the form (\ref{eq-ex2}), with structural constraints (\ref{eq-structconstex2}), may present a large variety of asymptotic behaviours without input, i.e. when $u=0$. This includes steady states, as shown Figures \ref{fig-ex2-1d} and \ref{fig-ex2-3d} , as well as limit cycles (not shown).\\
If $u$ is a scalar piecewise constant feedback, such that $u(a)$ satisfies (\ref{eq-solex2}) for $a\in\mathcal{A}^\star$, and $u(a)=0$ elsewhere, then there exists a unique, stable and globally attractive limit cycle.
\end{prop}
\begin{figure}
%\centering
%\epsfig{file=fig_ex1_oneCI.eps,height=7cm}
\centering
\includegraphics[width=8cm,height=8cm]{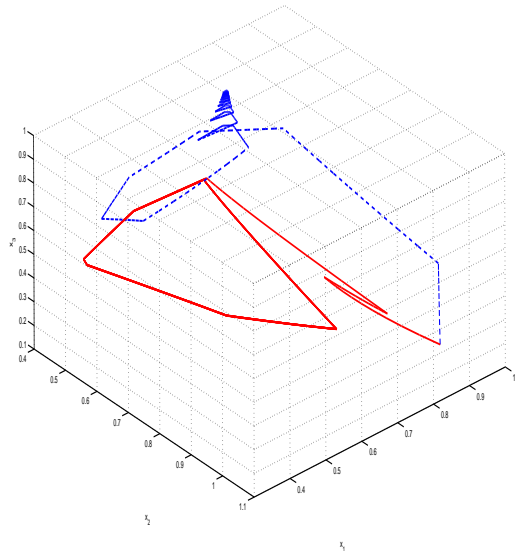}
\caption{\label{fig-ex2-1d}Blue, dashed line: without feedback control, spirals towards a fixed points. Red, plain line: with control, tends to a limit cycle. Common initial condition $(x_1,x_2,x_3)=(0.95,0.95,0.1)$. See parameters in Appendix \ref{sec-param-3}}
\end{figure}
\begin{figure}
%\centering
%\epsfig{file=fig_ex1_oneCI.eps,height=7cm}
\includegraphics[width=8cm,height=8cm]{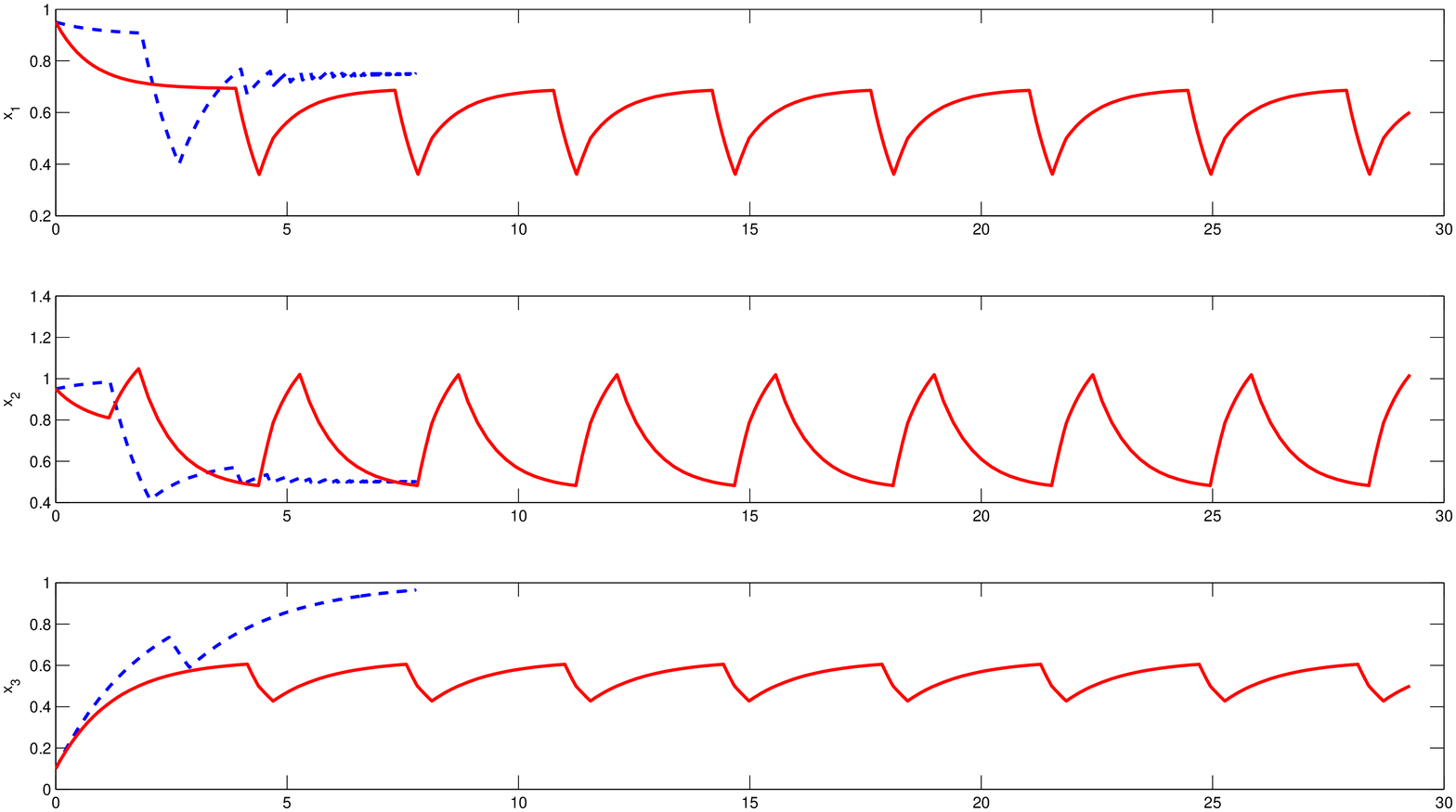}
\caption{\label{fig-ex2-3d}Same curves as Figure~\ref{fig-ex2-1d}, {\it vs} time.  The blue curve stops before the red one for the following reason. The numerical solutions are computed using the transition map, from wall to wall, and the uncontrolled trajectory crosses successive thresholds within time intervals tending to zero.}% See parameters in Appendix \ref{sec-param-3}}
\end{figure}

\section{Conclusion}
We have given, and illustrated by two examples, a control methodology to make unique stable limit cycles appear or disappear in hybrid PWA systems. The obtained feedback laws are termed qualitative control because they depend only on a qualitative abstraction of the original system; its transition graph.\\
Future work suggested by this study are mostly related to the question of dynamic feedback. Actually, the first example shows the effective possibility of using an additional variable to control a system, i.e. to design a controller system to be coupled to the original one. Moreover, the design of this dynamic feedback relied in a simple way on the static feedback problem. This technique should be formalised in more general terms, and applied to other examples in the future.

%two ingredients for theorem \ref{thm-main} and \ref{thm-negloop}: interaction structure (ie. focal point alignment) and cycle in $\mathsf{TG}$ with no escaping edge. Here we play on the second, assuming the first. One may also want to play on the first, and for example try to ensure the 'quasi-alignement' of focal points, and then invoke structural stability of the first return map.

\section{Acknowledgements}
J.-L.G. was partly funded by the LSIA-INRIA {\em Colage} and the ANR {\em MetaGenoReg} projects. E.F. was partly funded by the ANR-BBSRC {\em Flower Model} project.

%\balancecolumns

%APPENDICES are optional

\appendix
%%%Appendix A
\section{Parameter values}\label{sec-param}
\subsection{Parameters for Figure \ref{fig-ex1-2D}}\label{sec-param-1}
$$
\begin{array}{|c|c|c|c|c|c|c|c|}
\hline
K_1 &  K_2 & \gamma_1^0 & \gamma_1^1& \gamma_2^0 &\theta_1^1&\theta_1^2 & \theta_2^{1\phantom{\displaystyle A}} \\
\hline
0.9  &0.2 &1&1&0.3& 0.5 & 0.75 & 0.5\\
\hline
\end{array}
$$
Moreover the value $u(a)$ is computed as the middle-point of the interval defined by (\ref{eq-solex1}).

\subsection{Parameters for Figure \ref{fig-ex1-3D}}\label{sec-param-2}
$$
\begin{array}{|c|c|c|c|c|c|c|c|c|}
\hline
K_1 &  K_2 & \gamma_1^0 & \gamma_1^1& \gamma_2^0 &\theta_1^1&\theta_1^2 & \theta_2^{1\phantom{\displaystyle A}}&\theta_y  \\
\hline
0.9  &0.2 &1&1&0.3& 0.5 & 0.75 & 0.5 & 0.5\\
\hline
\end{array}
$$
To check the inequality in proposition \ref{prop-ex1}, we need to compute $ \left(\frac{\phi_1-\theta_1^2}{\phi_1-\theta_1^1}\right)^{\frac{1}{\gamma_1^0}}=0.375$. Then, the two values of $\gamma_y$ we have tested are $0.1$ and $1.7$, for which $\left(1-\gamma_y\theta_y\right)^{\frac{1}{\gamma_y}}$ is respectively close to $0.599$ (inequality satisfied) and $0.328$ (inequality violated).

\subsection{Parameters for Figures \ref{fig-ex2-1d} and \ref{fig-ex2-3d}}\label{sec-param-3}
$$
\begin{array}{|c|c|c|c|c|c|c|c|c|c|}
\hline
K_1 &  K_2^1 & K_2^3 & K_3 & \gamma_1^0 & \gamma_2^0& \gamma_3^0 &\gamma_i^1 &\theta_i^1&\theta_1^{2\phantom{\displaystyle A}}\\
\hline
0.9  &0.6 &1&0.5 & 1 & 1 & 0.5 & 1 & 0.5 & 0.75 \\
\hline
\end{array}
$$
where $i$ stands for all values in $\{1,2,3\}$
For these values, inequality (\ref{eq-solex2}) writes, term by term:
$$
\max\{0.2,0.2\}<u(a)<\min\{1,0.8,0.5\}
$$
and we have chosen $u(a)=0.3$ in the simulations.


\begin{thebibliography}{10}

\bibitem{msb06}
E.~Andrianantoandro, S.~Basu, D.~Karig, and R.~Weiss.
\newblock Synthetic biology: new engineering rules for an emerging discipline.
\newblock {\em Mol. Syst. Biol.}, 2, 2006.

\bibitem{batt2005vqm}
G.~Batt, C.~Belta, and R.~Weiss.
\newblock Model checking genetic regulatory networks with parameter
  uncertainty.
\newblock In {\em Hybrid systems: computation and control}, pages 61--75, 2007.

\bibitem{belta2006ccn}
C.~Belta and L.~Habets.
\newblock Controlling a class of nonlinear systems on rectangles.
\newblock {\em IEEE Transactions On Automatic Control}, 51(11):1749, 2006.

\bibitem{casey}
R.~Casey, H.~{de Jong}, and J.-L. Gouz\'e.
\newblock Piecewise-linear models of genetic regulatory networks: Equilibria
  and their stability.
\newblock {\em J. Math. Biol.}, 52(1):27--56, 2006.

\bibitem{bacillus}
H.~{de Jong}, J.~Geiselmann, G.~Batt, C.~Hernandez, and M.~Page.
\newblock Qualitative simulation of the initiation of sporulation in bacillus
  subtilis.
\newblock {\em Bull. Math. Biol.}, 66(2):261--300, 2004.

\bibitem{inria}
H.~{de Jong}, J.-L. Gouz\'e, C.~Hernandez, M.~Page, T.~Sari, and J.~Geiselmann.
\newblock Qualitative simulation of genetic regulatory networks using
  piecewise-linear models.
\newblock {\em Bull. Math. Biol.}, 66(2):301--340, 2004.

\bibitem{deveaux03}
Y.~Deveaux, A.~Peaucelle, G.~R. Roberts, E.~Coen, R.~Simon, Y.~Mizukami,
  J.~Traas, J.~A. Murray, J.~H. Doonan, and P.~Laufs.
\newblock The ethanol switch : a tool for tissue specific gene induction during
  plant development.
\newblock {\em Plant J.}, 36:918--930, 2003.

\bibitem{edwards}
R.~Edwards.
\newblock Analysis of continuous-time switching networks.
\newblock {\em Physica D}, 146:165--199, 2000.

\bibitem{edwardetc}
R.~Edwards, H.~Siegelmann, K.~Aziza, and L.~Glass.
\newblock Symbolic dynamics and computation in model gene networks.
\newblock {\em Chaos}, 11(1):160--169, 2001.

\bibitem{letnat1}
M.~B. Elowitz and S.~Leibler.
\newblock A synthetic oscillatory network of transcriptional regulators.
\newblock {\em Nature}, 403:335--338, 2000.

\bibitem{farcot06}
E.~Farcot.
\newblock Geometric properties of piecewise affine biological network models.
\newblock {\em J. Math. Biol.}, 52(3):373--418, 2006.

\bibitem{FG06a}
E.~Farcot and J.-L. Gouz\'e.
\newblock Periodic solutions of piecewise affine gene network models: the case
  of a negative feedback.
\newblock Research Report RR-6018, INRIA, 2006.
\newblock {\tt http://hal.inria.fr/inria-00112195/en/}.

\bibitem{FG08}
E.~Farcot and J.-L. Gouz\'e.
\newblock A mathematical framework for the control of piecewise-affine models
  of gene networks.
\newblock {\em Automatica}, 44(9):2326--2332, 2008.

\bibitem{ijss09}
E.~Farcot and J.-L. Gouz\'e.
\newblock Limit cycles in piecewise-affine gene network models with multiple
  interaction loops.
\newblock {\em Int. J. Syst. Sci.}, 2009.
\newblock to appear.

\bibitem{ab09}
E.~Farcot and J.-L. Gouz\'e.
\newblock Periodic solutions of piecewise affine gene network models: the case
  of a negative feedback loop.
\newblock {\em Acta Biotheoretica}, 57(4):429--455, 2009.

\bibitem{mfl05}
P.~Fran\c{c}ois and V.~Hakim.
\newblock Core genetic module: The mixed feedback loop.
\newblock {\em Phys. Rev. E}, 72:031908, 2005.

\bibitem{gedeon03}
T.~Gedeon.
\newblock Attractors in continuous time switching networks.
\newblock {\em Communications on Pure and Applied Analysis (CPAA)},
  2(2):187--209, 2003.

\bibitem{glasstopkin}
L.~Glass.
\newblock Combinatorial and topological methods in nonlinear chemical kinetics.
\newblock {\em J. Chem. Phys.}, 63:1325--1335, 1975.

\bibitem{glasspastern2}
L.~Glass and J.~S. Pasternack.
\newblock Stable oscillations in mathematical models of biological control
  systems.
\newblock {\em J. Math. Biol.}, 6:207--223, 1978.

\bibitem{gouzesari}
J.-L. Gouz\'e and T.~Sari.
\newblock A class of piecewise linear differential equations arising in
  biological models.
\newblock {\em Dynamical Systems}, 17:299--316, 2003.

\bibitem{grilly07}
C.~Grilly, J.~Stricker, W.~L. Pang, M.~R. Bennett, and J.~Hasty.
\newblock A synthetic gene network for tuning protein degradation in {\it
  saccharomyces cerevisiae}.
\newblock {\em Mol. Syst. Biol.}, 3:127, 2007.

\bibitem{HabJVS2004}
L.~Habets and J.~{van Schuppen}.
\newblock A control problem for affine dynamical systems on a full-dimensional
  polytope.
\newblock {\em Automatica}, 40:21-- 35., 2004.

\bibitem{iglesiasingalls}
P.~A. Iglesias and B.~P. Ingalls, editors.
\newblock {\em Control Theory and Systems Biology}.
\newblock MIT Press, 2009.

\bibitem{pnas-22-04}
H.~Kobayashi, M.~Kaern, M.~Araki, K.~Chung, T.~S. Gardner, C.~R. Cantor, and
  J.~J. Collins.
\newblock Programmable cells: interfacing natural and engineered gene networks.
\newblock {\em Proc. Natl. Acad. Sci. U.S.A.}, 101(22):8414--9, 2004.

\bibitem{Edw2009}
L.~Lu and R.~Edwards.
\newblock Structural principles for periodic orbits in glass networks.
\newblock {\em Journal of Mathematical Biology}, 2009.
\newblock DOI 10.1007/s00285-009-0273-8, to appear.

\bibitem{periodsol}
T.~Mestl, E.~Plahte, and S.~W. Omholt.
\newblock Periodic solutions of piecewise-linear differential equations.
\newblock {\em Dyn. Stab. Syst.}, 10(2):179--193, 1995.

\bibitem{plahmestl98}
E.~Plahte, T.~Mestl, and S.~W. Omholt.
\newblock A methodological basis for description and analysis of systems with
  complex switch-like interactions.
\newblock {\em J. Math. Biol.}, 36:321--348, 1998.

\bibitem{shimizu02}
S.~{Shimizu-Sato}, E.~Huq, J.~Tepperman, and P.~H. Quail.
\newblock A light-switchable gene promoter system.
\newblock {\em Nat. Biotechnol.}, 20(10):1041--1044, 2002.

\bibitem{snoussi}
E.~H. Snoussi.
\newblock Qualitative dynamics of piecewise-linear differential equations: a
  discrete mapping approach.
\newblock {\em Dyn. Stab. Syst.}, 4(3-4):189--207, 1989.

\bibitem{sontag05}
E.~D. Sontag.
\newblock Molecular systems biology and control.
\newblock {\em Eur. J. Control}, 11((4-5)):396--435, 2005.

\bibitem{wyke06}
S.~Wyke and M.~Tisdale.
\newblock Induction of protein degradation in skeletal muscle by a phorbol
  ester involves upregulation of the ubiquitin-proteasome proteolytic pathway.
\newblock {\em Life Sciences}, 78(25):2898 -- 2910, 2006.

\end{thebibliography}
\end{document}